\documentclass[11pt]{amsart}
\usepackage[left=2cm,right=2cm,top=2cm,bottom=2cm]{geometry}
\usepackage{amsmath,amssymb,amsthm,mathrsfs}
\usepackage[utf8]{inputenc}
\usepackage[T1]{fontenc}
\usepackage[english]{babel}
\usepackage{indentfirst}
\usepackage{tikz}
\usetikzlibrary{cd}
\usepackage{cleveref}

\def\Ih{\mathfrak{h}}
\newtheorem{definition}{Definition}[section]
\theoremstyle{definition}
\newtheorem{remark}[definition]{Remark}
\theoremstyle{definition}
\newtheorem{proposition}[definition]{Proposition}
\theoremstyle{definition}

\theoremstyle{definition}
\newtheorem{theorem}[definition]{Theorem}
\theoremstyle{definition}
\newtheorem{corollary}[definition]{Corollary}
\theoremstyle{definition}

\begin{document}

\title{Ext groups in the category of bimodules over a simple Leibniz algebra} 

\author{Jean Mugniery}
\address{Laboratoire de Mathématiques, Université Louvain-la-Neuve,
Chemin du Cyclotron, 1348 Louvain-la-Neuve, Belgique}
\email{jean.mugniery@uclouvain.be}

\author{Friedrich Wagemann}
\address{Laboratoire de math\'ematiques Jean Leray, UMR 6629 du CNRS, Universit\'e
de Nantes, 2, rue de la Houssini\`ere, F-44322 Nantes Cedex 3, France}
\email{wagemann@math.univ-nantes.fr}

\maketitle

\subjclass[2010]{Primary 17A32; Secondary 17B56}

\keywords{Leibniz cohomology, Chevalley-Eilenberg cohomology, change of rings spectral sequence, cohomology vanishing, semi-simple Leibniz algebra, Whitehead theorem, Ext category of Leibniz bimodules}

\begin{abstract}
In this article, we generalize Loday and Pirashvili's \cite{LP} computation of the $Ext$-category of Leibniz bimodules for a simple Lie algebra to the case of a simple (non Lie) Leibniz algebra. Most of the arguments generalize easily, while the main new ingredient is the Feldvoss-Wagemann's cohomology vanishing theorem for semi-simple Leibniz algebras.   
\end{abstract}

\section*{Introduction}

The goal of this article is to present a new result in the theory of representations of Leibniz algebras, namely to compute Ext groups between finite dimensional simple bimodules of a simple (non Lie) Leibniz algebra over a field of characteristic $0$. Leibniz algebras are a generalization of Lie algebras, discovered by A. Bloh in the 1960s, where one does not require the bracket to be antisymmetric. They were rediscovered by J.-L. Loday \cite{L} in the 1990s when, while trying to lift the boundary operator of Chevalley-Eilenberg homology from $d : \Lambda^{n} \mathfrak{g} \longrightarrow \Lambda^{n-1} \mathfrak{g}$ to $\tilde{d} : \mathfrak{g}^{\otimes n} \longrightarrow \mathfrak{g}^{\otimes n-1}$, he noticed that the only property needed to show that $\tilde{d} \circ \tilde{d} = 0$ was the Leibniz identity of the bracket, that is:
\begin{equation*}
\left[ x, \left[ y, z\right] \right] = \left[ \left[ x, y \right] , z\right] + \left[ y, \left[ x, z \right] \right] \,\,\,\,\forall x, y, z \in \mathfrak{g}
\end{equation*}
\noindent
(see \cite{L} for a survey of the subject).

J.-L. Loday and T. Pirashvili studied the Leibniz representations of semi-simple Lie algebras in \cite{LP}, and established in that paper the following theorem (Theorem 3.1):

\begin{theorem} \label{thm3.1}
Let $\mathfrak{g}$ be a finite dimensional simple Lie algebra over a field of characteristic $0$, and let $L(\mathfrak{g})$ denote the category of finite dimensional Leibniz representations of $\mathfrak{g}$.
The simple objects in $L(\mathfrak{g})$ are exactly the representations of the form $M^{a}$ and $N^{s}$, where $M$, and $N$ are simple right $\mathfrak{g}$-modules.
All groups $Ext^{2}_{UL(\mathfrak{g})}(M, N)$ between simple finite dimensional representations $M$, $N$ are zero, except $Ext^{2}_{UL(\mathfrak{g})}(\mathfrak{g}^{s}, \mathfrak{g}^{a})$ which is one-dimensional. Moreover, 
\begin{equation*}
Ext^{1}_{UL(\mathfrak{g})}(M^{s}, N^{a}) \simeq Hom_{U(\mathfrak{g})}(M, \hat{N})
\end{equation*}
where
\begin{equation*}
\hat{N} = Coker( h : N \longrightarrow Hom(\mathfrak{g}, N)), h(n)(x) = [n, x]
\end{equation*}
and all other groups $Ext^{1}_{UL(\mathfrak{g})}(M, N)$ between simple finite dimensional representations $M$, $N$ are zero.
\end{theorem}

This theorem shows in particular that, contrary to the representations of semi-simple Lie algebras, the category of Leibniz bimodules of a semi-simple Lie algebra is not semi-simple.

The aim of our article is to generalize this result to Leibniz representations of simple (non Lie) Leibniz algebras by closely following \cite{LP}, and making the adequate changes whenever necessary. The key in doing so is a theorem of J. Feldvoss and F. Wagemann, namely Theorem 4.2 of \cite{FelWag}, assuring the vanishing of Leibniz cohomology needed in the proof of our main theorem:

\begin{theorem} \label{fw}
Let $\Ih$ be a finite-dimensional semisimple left Leibniz algebra over a field of characteristic zero, and let $M$ be a finite-dimensional $\Ih$-bimodule. Then $HL^{n}(\Ih, M) = 0$ for every integer $n \geq 2$, and if $M$ is symmetric, then $HL^{n}(\Ih, M) = 0$ for every integer $n \geq 1$.
\end{theorem}

Interestingly, this result represents a continuation of other vanishing theorems. First Whitehead's Theorem giving the vanishing of Chevalley-Eilenberg cohomology of a semi-simple Lie algebra with values in a finite dimensional $\mathfrak{g}$-module whose invariants are trivial. Then P. Ntolo in \cite{N} and T. Pirashvili in \cite{P} independently proved results about Leibniz (co)homology of Lie algebras, while the authors of \cite{FelWag} proved the vanishing of Leibniz cohomology of semi-simple Leibniz algebras.

This allows us to prove the following theorem which is the main theorem of this article. Recall that  $\mathfrak{Leib(h)}$ denotes the two-sided ideal generated by the elements $\left[ x, x \right]$ for $x \in \Ih$ and  that ${\mathfrak h}_{\rm Lie}:={\mathfrak h}/\mathfrak{Leib(h)}$ denotes the canonical quotient Lie algebra associated to a Leibniz algebra ${\mathfrak h}$. Both of them are in particular left $U{\mathfrak h}_{\rm Lie}$-modules and $\mathfrak{Leib(h)}^{\star}$ and $\Ih_{Lie}^{\star}$ are the corresponding dual left modules.  

\begin{theorem}
Let $\Ih$ be a finite dimensional simple Leibniz algebra over a field of characteristic zero $k$. All groups $Ext_{UL(\Ih)}^{2}(M, N)$ between simple finite dimensional $\Ih$-bimodules are zero, except $Ext^{2}_{UL(\Ih)}(\mathfrak{M}^{s}, \mathfrak{N}^{a})$, with $\mathfrak{M} \in \{ \mathfrak{Leib(h)}^{\star}, \Ih_{Lie}^{\star} \}$ and $\mathfrak{N} \in \{ \mathfrak{Leib(h)}, \Ih_{Lie} \}$  which is one dimensional.\\
Moreover, we have that:
\begin{itemize}
\item $Ext^{1}_{UL(\Ih)}(\mathfrak{M}^{s}, k)$, and $Ext^{1}_{UL(\Ih)}(k, \mathfrak{N}^{a})$ are one dimensional, for $\mathfrak{M}\in \{ \mathfrak{Leib(h)}^{\star}, \Ih_{Lie}^\star \}$ and $\mathfrak{N} \in \{ \mathfrak{Leib(h)}, \Ih_{Lie} \}$;
\item $Ext^{1}_{UL(\Ih)}(M^{s}, N^{a}) \simeq Hom_{U(\Ih_{Lie})}(M, \widehat{N})$, where
\begin{equation*}
\widehat{N} := Coker(h:\; \; N \longrightarrow Hom(\Ih, N))\; \; \; \; h(n)(x):= \left[ x, n\right]_{L}
\end{equation*}
\item All other groups $Ext^{1}_{UL(\Ih)}(M, N)$ between simple finite-dimensional $\Ih$-bimodules $M$ and $N$ are zero.
\end{itemize}  
\end{theorem}

We see that, when we do not restrict ourselves to Lie algebras, there are more non-trivial Ext groups. Moreover, as a corollary of this theorem one can show that the Ext dimension of the category of finite dimensional bimodules over a semi-simple Leibniz algebra is again $2$.

We work in our article with left Leibniz algebras. All preliminary results about Leibniz algebras are due to Loday and Pirashvili and shown in the framework of right Leibniz algebras. References where the corresponding results are shown for left Leibniz algebras include \cite{C} and \cite{F}.    

\vspace{.5cm}

\noindent{\bf Acknowledgements:}

The authors would like to thank A. Djament for pointing out a simpler way to show that the Ext in the category of finite dimensional bimodules over a simple Leibniz algebra are well defined.

\setcounter{section}{0}
\section{Leibniz Algebras}

In this section, we introduce the objects in which we are interested, as well as some of their basic properties. All of this material is due to Loday and Pirashvili. For more results on Leibniz algebras as non-associative algebras see \cite{F}, and see \cite{loday} for results about their (co)homology.

\begin{definition}
A (left) \textbf{Leibniz algebra} over a field k is a vector space $\Ih$ equipped with a bilinear map:
\[
\left[ -,-\right] : \Ih \times \Ih \longrightarrow \Ih
\]
called Leibniz bracket, that satisfies the (left) Leibniz identity:
\begin{equation}
\left[ x, \left[ y, z\right] \right] = \left[ \left[ x, y \right] , z\right] + \left[ y, \left[ x, z \right] \right] \forall x, y, z \in \Ih
\end{equation}
\end{definition}

With this definition, we see that Leibniz algebras are indeed a generalization of Lie algebras, as it is not difficult to check that if we impose the anticommutativity of the bracket, the Jacobi and Leibniz identities are equivalent. 

\begin{remark}
We can also define a right Leibniz algebra by asking our bracket to satisfy the right Leibniz identity instead: $  \left[ \left[ x, y \right] , z\right] =\left[ \left[ x, z \right] , y\right] +  \left[ x, \left[ y, z\right] \right] $, but we will only be concerned with left Leibniz algebras.
\end{remark}

For every Leibniz algebra $\Ih$, we have a short exact sequence:
\begin{equation}
0 \longrightarrow \mathfrak{Leib(h)} \longrightarrow \Ih \longrightarrow \Ih _{Lie} \longrightarrow 0
\end{equation}
where $\mathfrak{Leib(h)}$ is the $\textit{Leibniz kernel}$ of $\Ih$, that is the two-sided ideal generated by the elements $\left[ x, x \right]$ for $x \in \Ih$; and $\Ih _{Lie} := \Ih / \mathfrak{Leib(h)}$.
$\Ih _{Lie}$ is a Lie algebra, called the \textit{canonical Lie algebra associated to $\Ih$ }

\begin{definition}
A left Leibniz algebra is called simple if 0, $\mathfrak{Leib(h)}$, and $\Ih$ are the only two sided ideals of $\Ih$, and $\mathfrak{Leib(h)} \subsetneqq \left[ \Ih, \Ih \right]$.
\end{definition}

This is not the only definition of simplicity: One can also only require that $0$ and $\Ih$ are the only ideals of $\Ih$, but with this definition, all simple Leibniz algebras are in fact Lie algebras, see the beginning of section 7 of \cite{F}, and the references therein.
We will also need the Proposition 7.2 of \cite{F}, namely:

\begin{proposition}
If $\Ih$ is a simple Leibniz algebra, then $\Ih_{Lie}$ is a simple Lie algebra and $\mathfrak{Leib(h)}$ is  a simple $\Ih_{Lie}$-module.
\end{proposition}

We now give the definition of the notion of Leibniz modules and bimodules.

\begin{definition} \label{def}
Let $\Ih$ be a Leibniz algebra. An $\Ih$-bimodule is a vector space $M$ over $k$ equipped with two bilinear maps:
\[ 
\left[ -,-\right] _{L} : \Ih \times M \longrightarrow M
\]
and
\[
\left[ -,-\right] _{R} : M \times \Ih \longrightarrow M
\]
which satisfy the following relations $\forall x, y \in \Ih,  \forall m \in M$:

\begin{equation}
\tag{LLM}
\left[ x, \left[ y, m\right]_{L} \right]_{L} = \left[ \left[ x, y \right] , m\right]_{L} + \left[ y, \left[ x, m \right]_{L} \right]_{L} 
\end{equation}
\begin{equation}
\tag{LML}
\left[ x, \left[ m, y\right]_{R} \right]_{L} = \left[ \left[ x, m \right]_{L} , y\right]_{R} + \left[ m, \left[ x, y \right] \right]_{R} 
\end{equation}
\begin{equation}
\tag{MLL}
\left[ m, \left[ x, y\right] \right]_{R} = \left[ \left[ m, x \right]_{R} , y\right]_{R} + \left[ x, \left[ m, y \right]_{R} \right]_{L} 
\end{equation}
\end{definition}

We define a \textbf{left $\Ih$-module} as being a vector space $M$ over $k$ equipped with a bilinear map:
\[
\left[ -,-\right] _{L} : \Ih \times M \longrightarrow M
\]
satisfying the relation (LLM) of \Cref{def}.

\begin{definition}
Let $\Ih$ be a Leibniz algebra, and $M$ a Leibniz bimodule.\\
If
\[
\left[x, m \right] _{L} = - \left[ m, x \right] _{R} \: \forall x \in \Ih, \forall m \in M
\] then M is said to be \textbf{symmetric} and denoted $M^{s}$.\\
If 
\[ \left[ m, x \right] _{R} = 0 \; \forall x \in \Ih, \forall m \in M
\] then M is said to be \textbf{antisymmetric} and denoted $M^{a}$.\\
If M is both symmetric and antisymmetric, then M is \textbf{trivial}.
\end{definition}

For every $\Ih$-bimodule $M$, there is a short exact sequence of $\Ih$-bimodules:
\begin{equation}
0 \longrightarrow M_{0} \longrightarrow M \longrightarrow M/M_{0} \longrightarrow 0
\end{equation}
where $M_{0} = Span_{k}( \left[ x, m \right]_{L} + \left[ m, x \right]_{R})$, see (1.10) of \cite{loday}. 
Note that by construction $M/M_{0}$ is a symmetric $\Ih$-bimodule, and that $M_{0}$ is an antisymmetric $\Ih$-bimodule. Moreover, if we consider $\Ih$ as an $\Ih$-bimodule using the adjoint action, then the short exact sequences (2) and (3) coincide.

If $M$ is an $\Ih$-bimodule (in fact this works even when $M$ is only a left $\Ih$-module), then it has a natural $\Ih _{Lie}$-module structure (in the Lie sense). Indeed one can define a left action of $\Ih _{Lie}$ as follows:
\[
\begin{split}
\Ih _{Lie} \times M & \longrightarrow M \\
(\bar{x}, m ) & \longmapsto \left[ x, m \right] _{L}
\end{split}
\]
Conversely, if $M$ is an $\Ih_{Lie}$-module, there are two natural ways to see it as an $\Ih$-bimodule. We first see it as a left $\Ih$-module via the projection $\Ih \longrightarrow \Ih_{Lie}$, and then we impose our right action to be either trivial, or to be the opposite of the left action, yielding respectively an antisymmetric bimodule, or a symmetric one.
Knowing this we can state the following Theorem (due to Loday-Pirashvili \cite{LP}):

\begin{theorem} 
The simple objects in the category of $\Ih$-bimodules of finite dimension are exactly the modules of the form $M^{a}$ and $M^{s}$, where $M$ is a finite dimensional simple $\Ih_{Lie}$-module.
\end{theorem}

The proof follows easily from the existence of the short exact sequence (3), and the fact that $\mathfrak{Leib(h)}$ acts trivially from the left (i.e. is contained in the left center).

We now introduce the notion of the universal enveloping algebra of a Leibniz algebra, see (2.1) of \cite{loday} (but note that the authors work with \textit{right} Leibniz algebras. A reference for left Leibniz algebras is \cite{C}.).

\begin{definition}
Let $\Ih$ be a Leibniz algebra. Given two copies $\Ih^{l}$ and $\Ih^{r}$ of $\Ih$ generated respectively by the elements $l_{x}$ and $r_{x}$ for $x \in \Ih$, we define the universal enveloping algebra of $\Ih$ as the unital associative algebra:
\[ 
UL(\Ih ) := T( \Ih ^{l} \oplus \Ih ^{r} ) / \mathfrak{I}
\]
where $T( \Ih ^{l} \oplus \Ih ^{r} ) := \displaystyle \bigoplus _{n = 0} ^{\infty}( \Ih ^{l} \oplus \Ih ^{r} )^{\otimes n}$ is the tensor algebra of $ \Ih ^{l} \oplus \Ih ^{r}$ and $\mathfrak{I} $ is the two-sided ideal of $ \Ih$ generated by the elements : 
\[
\begin{split}
l_{ \left[x, y\right] } - l_{x} \otimes l_{y} + l_{y} \otimes l_{x} \\
r_{ \left[x, y\right] } - l_{x} \otimes r_{y} + r_{y} \otimes l_{x} \\
r_{y} \otimes (l_{x} + r_{x})
\end{split}
\]
\end{definition}

For a Lie algebra $\mathfrak{g}$, there is an equivalence between being a $\mathfrak{g}$-module and being a $U(\mathfrak{g})$-module, where $U(\mathfrak{g})$ is the universal enveloping algebra of the Lie algebra $\mathfrak{g}$. The following theorem allows us to establish the same kind of connection between the structure of $\Ih$-bimodule and left $UL(\Ih)$-module (due to Loday-Pirashvili \cite{loday}. See also \cite{C} for a proof in the framework of left Leibniz algebras).

\begin{theorem} \label{equi}
Let $\Ih$ be a Leibniz algebra. There is an equivalence of categories between the category of $\Ih$-bimodules and the category of $UL(\Ih)$-modules.
\end{theorem}

For a proof see Theorem (2.3) of \cite{loday} (note once again that the authors work with \textit{right} Leibniz algebras). This actually tells us, given one of the two structures, how to obtain the other: The action of $l_{x}$ corresponds to the left action $\left[ x, - \right]_{L}$ while the action $r_{y}$ corresponds to the right action $\left[ -, y \right]_{R}$.

Another nice property of this universal enveloping algebra is that we can establish a connection between $U(\Ih _{Lie})$-modules and $UL(\Ih)$-modules. To this end we define the following algebras homomorphisms:
\[
\begin{split}
d_{0} : UL(\Ih ) & \longrightarrow U(\Ih _{Lie}) \\
d_{0} (l_{x} ) & = \bar{x} \\
d_{0} (r_{x} ) & = 0
\end{split}
\]
and: 
\[
\begin{split}
d_{1} : UL(\Ih ) & \longrightarrow U(\Ih _{Lie}) \\
d_{1} (l_{x} ) & = \bar{x} \\
d_{1} (r_{x} ) & = - \bar{x}
\end{split}
\]

With these, given a $U(\Ih _{Lie})$-module, we can view it as a $UL(\Ih)$-module either via $d_{0}$ or via $d_{1}$. The former gives an antisymmetric $\Ih$-bimodule, while the latter gives a symmetric $\Ih$-bimodule. Moreover, since they are surjective (their image contains the generators of $U(\Ih_{Lie})$), this allows us to consider $U(h_{Lie})$ as the quotient $UL(\Ih)/Ker(d_{i})$ for $i \in \{ 0, 1\}$.

Being a generalization of Lie algebras, Leibniz algebras are equiped with a generalization of Chevalley-Eilenberg cohomology, namely Leibniz cohomology which was discovered by Loday. Let $\Ih$ be a Leibniz algebra, and $M$ be an $\Ih$-bimodule. We give the (left version of the) definition of the cochain complex
\begin{equation*}
CL^{n}(\Ih , M), dL^{n} \} _{n \geq 0 }
\end{equation*} 
from \cite{loday} (1.8) (a reference for this left version is \cite{FelWag}), namely:
\[
\begin{split}
CL^{n}(\Ih , M) & = Hom(\Ih ^{\otimes n}, M) \\
dL^{n} : CL^{n}(\Ih, M) & \longrightarrow CL^{n+1}(\Ih, M)
\end{split}
\]
with:
\[
\begin{split}
dL^{n} \omega (x_{0}, ..., x_{n}) = & \sum_{i=0}^{n-1} (-1)^{i} \left[ x_{i}, \omega (x_{0}, ..., \hat{x_{i}}, ..., x_{n}) \right] _{L} + (-1)^{n-1} \left[ \omega (x_{0}, ..., x_{n-1}), x_{n} \right] _{R} \\
& + \sum_{0 \leq i < j \leq n} (-1)^{i+1} \omega (x_{0}, ..., \hat{x_{i}}, ..., x_{j-1}, \left[ x_{i}, x_{j} \right] , x_{j+1}, ..., x_{n} ) 
\end{split}
\]

\begin{definition}
Let $\Ih$ be a Leibniz algebra, and $M$ be an $\Ih$-bimodule. The cohomology of $\Ih$ with coefficients in M is the cohomology of the cochain complex  $\{ CL^{n}(\Ih , M), dL^{n} \} _{n \geq 0 }$.
\[
HL^{n}(\Ih, M) = H^{n}(\{ CL^{n}(\Ih , M), dL^{n} \} _{n \geq 0 }) \; \; \forall n \geq 0
\]
\end{definition}

\begin{remark}
By definition $CL^{0}(\Ih, M) = M$ and $dL^{0}m(x) = -\left[m, x, \right]_{R}$. Therefore, we have:
\[
HL^{0}(\Ih, M) = \{ m \in M, \; \; \left[m, x \right]_{R} = 0 \; \; \forall x \in \Ih \}
\]
This is the submodule of right invariants. Note that if $M$ is antisymmetric, then $HL^{0}(\Ih, M) = M$.
\end{remark}

\section{Ext in the category of Leibniz bimodules}

We are now interested in computing the Ext groups in the category of $\Ih$-bimodules. From now on, we will consider a finite-dimensional left Leibniz algebra $\Ih$ over a field of characteristc zero $k$. The definition of the morphisms $d_{0}$ and $d_{1}$, together with the change of rings spectral sequence constructed in the subsections 1 to 4 of Chapter XVI from \cite{CE}, yield the following spectral sequences:

\begin{align*}
E^{pq}_{2} = Ext_{U(\Ih_{Lie})}^{p}\left( Y, Ext_{UL(\Ih)}^{q}\left(U(\Ih_{Lie})^{a}, X\right) \right) &  \Longrightarrow Ext^{p+q}_{UL(\Ih)}(Y^{a}, X) && \text{(S1)} \\
E^{pq}_{2} = Ext_{U(\Ih_{Lie})}^{p}\left( Z, Ext_{UL(\Ih)}^{q}\left(U(\Ih_{Lie})^{s}, X\right) \right) &  \Longrightarrow Ext^{p+q}_{UL(\Ih)}(Z^{s}, X) && \text{(S2)}
\end{align*}

where $X$ is an $\Ih$-bimodule, and $Y$ and $Z$ are left $\Ih$-modules.\\
For a Lie algebra $\mathfrak{g}$, we have the following isomorphism $Ext_{U(\mathfrak{g})}^{*}(M, N) \simeq H^{*}\left( \mathfrak{g}, Hom(M, N) \right)$, which we can use to rewrite (S1) and (S2) as:

\begin{align*}
E^{pq}_{2} = H^{p}\left(\Ih_{Lie}, Hom(Y, Ext_{UL(\Ih)}^{q}\left(U(\Ih_{Lie})^{a}, X\right) \right) &  \Longrightarrow Ext^{p+q}_{UL(\Ih)}(Y^{a}, X) && \text{(S1)}\\
E^{pq}_{2} = H^{p}\left( \Ih_{Lie}, Hom(Z, Ext_{UL(\Ih)}^{q}\left(U(\Ih_{Lie})^{s}, X\right) \right) &  \Longrightarrow Ext^{p+q}_{UL(\Ih)}(Z^{s}, X) && \text{(S2)}
\end{align*}

Moreover, by Theorem (3.4) of \cite{loday}, we have an isomorphism
\begin{equation*}
Ext^{*}_{UL(\Ih)}(U(\Ih_{Lie})^{a}, X) \simeq HL^{*}(\Ih, X)
\end{equation*}
This isomorphism also holds for the left framework as is easily shown by constructing the left version of the non-commutative Koszul complex of \cite{loday}. 

This gives us the proposition:

\begin{proposition} \label{sspec}
Let $\Ih$ be a Leibniz algebra, let $X$ be an $\Ih$-bimodule, and $Y$ and $Z$ be left $\Ih$-modules. There are two spectral sequences:
\begin{align*}
E^{pq}_{2} = H^{p}\left( \Ih_{Lie}, Hom(Y, HL^{q}(\Ih, X)) \right) &  \Longrightarrow Ext^{p+q}_{UL(\Ih)}(Y^{a}, X) && \text{(S1)} \\
E^{pq}_{2} = H^{p}\left(\Ih_{Lie}, Hom \left( Z, Ext^{q}_{UL(\Ih)}(U(\Ih_{Lie})^{s}, X) \right) \right) &  \Longrightarrow Ext^{p+q}_{UL(\Ih)}(Z^{s}, X) && \text{(S2)}
\end{align*}
\end{proposition}

In the previous proposition, we were able to identify $Ext^{*}_{UL(\Ih)}(U(\Ih_{Lie})^{a}, X)$ to the Leibniz cohomology $HL^{*}(\Ih, X)$. What about $Ext^{*}_{UL(\Ih)}(U(\Ih_{Lie})^{s}, X)$? The following result will give a generalization of Proposition 2.3 of \cite{LP}, in order to give a relation between $Ext^{*}_{UL(\Ih)}(U(\Ih_{Lie})^{s}, X)$ and Leibniz cohomology. In order to do so, we will have to introduce a shift in the homological degree which will be responsible for nontrivial Ext groups in what will follow.

\begin{proposition} \label{prop23}
Let $\Ih$ be a Leibniz algebra, and $M$ be an $\Ih$-bimodule. There are isomorphisms:
\begin{align*}
Ext_{UL(\Ih)}^{i+1}\left( U(\Ih _{Lie})^{s}, M \right) & \simeq Hom(\Ih, HL^{i}(\Ih, M)) && \text{for i > 0} \\
& \simeq Coker (f) && \text{for i = 0} \\
& \simeq Ker (f) && \text{for i = -1} 
\end{align*}
where $ f : M \longrightarrow Hom(\Ih, HL^{0}(\Ih, M))$ is given by:
\[
f(m)(h) = \left[ h, m \right] _{L} + \left[ m, h \right] _{R} \; \; \; \forall h \in \Ih, \; \forall m \in M
\]
\end{proposition} 

\begin{proof}
Let $M$ be an $\Ih$-bimodule, and
\[
\begin{split}
f : M & \longrightarrow Hom(\Ih, HL^{0}(\Ih, M)) \\
f(m)(h) & = \left[ h, m \right] _{L} + \left[ m, h \right] _{R}
\end{split}
\]
We first want to show that $Ext^{0}_{UL(\Ih)}(U(\Ih_{Lie})^{s}, M) = Ker (f)$. But by definition
\begin{equation*}
Ext^{0}_{UL(\Ih)}(U(\Ih_{Lie})^{s}, M) = Hom_{UL(\Ih)}(U(\Ih_{Lie})^{s}, M)
\end{equation*}
We then define the map:
\[
\begin{split}
ev : Hom_{UL(\Ih)}(U(\Ih_{Lie})^{s}, M) & \longrightarrow M \\
\varphi & \longmapsto \varphi(1)
\end{split}
\]
which is an isomorphism onto $Ker(f)$, of inverse:
\[
\begin{split}
\mu : Ker (f) & \longrightarrow Hom_{UL(\Ih)}(U(\Ih_{Lie})^{s}, M) \\
m & \longmapsto \varphi_{m} : (1 \mapsto m)
\end{split}
\]
This gives the degree zero equality of the proposition.\\
We now want to show that $Ext^{1}_{UL(\Ih)}(U(\Ih_{Lie})^{s}, M) = Coker (f)$.\\
Consider $UL(\Ih) \otimes \Ih$ as a left $UL(\Ih)$-module with the following action $\forall x \in \Ih, \forall r, s \in UL(\Ih)$:
\[
s.(r \otimes x) = sr \otimes x
\]

Define a homomorphism of left $UL(\Ih)$-modules by:
\[
\begin{split}
f_{1} : UL(\Ih) \otimes \Ih & \longrightarrow UL(\Ih) \\
1 \otimes h & \longmapsto l_{h} + r_{h}
\end{split}
\]

Then $f_{1}$ factors through $f_{2} : U(\Ih_{Lie})^{a} \otimes \Ih \longrightarrow UL(\Ih)$. Indeed we have the following commutative diagram:

\begin{equation}
\tag{D}
\begin{tikzcd}
UL(\Ih) \otimes \Ih \arrow[r, "f_{1}"] \arrow[d, "d_{0} \otimes id"'] & UL(\Ih) \\
U(\Ih_{Lie})^{a} \otimes \Ih \arrow[d] \arrow[ru, "f_{2}"'] \\
0
\end{tikzcd}
\end{equation}

and define $f_{2}(d_{0}(x) \otimes h) := f_{1}(x \otimes h)$ which is well-defined: If $x, y \in UL(\Ih)$ are such that $d_{0}(x) = d_{0}(y)$, then $f_{1}(x \otimes h) = f_{1}(y \otimes h)$. Indeed if $x-y \in Ker (d_{0})$, then $x = y + \bar{z}$ with $\bar{z} \in \langle r_{z}, z \in \Ih \rangle$. Therefore, the relation $r_{y}(l_{x} + r_{x}) = 0$ in $UL(\Ih)$ implies that  $f_{1}(x \otimes h) = f_{1}(y \otimes h)$. \\
We claim that $f_{2}$ is injective. This follows from the diagram (D) and the fact that $Ker (f_{1}) = Ker (d_{0} \otimes id)$. This therefore gives us the following short exact sequence:
\begin{equation*}
\begin{tikzcd}
0 \arrow[r] & U(\Ih_{Lie})^{a} \otimes \Ih \arrow[r, "f_{2}"] & UL(\Ih) \arrow[r] & Coker f_{2} \arrow[r] & 0
\end{tikzcd}
\end{equation*}
But by construction, $Im(f_{2})$ is the left ideal $\langle l_{x} + r_{x} \;|\;x \in \Ih \rangle$, which is equal to $Ker(d_{1})$ (see Section 1). This implies that $Coker(f_{2})$ is the quotient $UL(\Ih)/Ker(d_{1})$, that is $Im(d_{1})$, and the short exact sequence above becomes:
\begin{equation*}
\begin{tikzcd}
0 \arrow[r] & U(\Ih_{Lie})^{a} \otimes \Ih \arrow[r, "f_{2}"] & UL(\Ih) \arrow[r] & U(\Ih_{Lie})^{s} \arrow[r] & 0
\end{tikzcd}
\end{equation*}
This short exact sequence yields the following long exact sequence in cohomology: 
\[
\begin{split}
0 & \longrightarrow Hom_{UL(\Ih)}(U(\Ih_{Lie})^{s}, M) \longrightarrow Hom_{UL(\Ih)}(UL(\Ih), M) \longrightarrow Hom_{UL(\Ih)}(U(\Ih_{Lie})^{a} \otimes \Ih, M) \\
& \longrightarrow Ext^{1}_{UL(\Ih)}(U(\Ih_{Lie})^{s}, M) \longrightarrow Ext^{1}_{UL(\Ih)}(UL(\Ih), M) \longrightarrow Ext^{1}_{UL(\Ih)}(U(\Ih_{Lie})^{a} \otimes \Ih, M) \\
& \longrightarrow Ext^{2}_{UL(\Ih)}(U(\Ih_{Lie})^{s}, M) \longrightarrow\ldots
\end{split}
\]
Now, by noticing the obvious identification $Hom_{UL(\Ih)}(UL(\Ih), M) = M$, and the fact that, $UL(\Ih)$ being a free $UL(\Ih)$-module, it is projective and therefore $Ext^{1}_{UL(\Ih)}(UL(\Ih), M) = 0$, we can extract the following exact sequence:
\[
0 \rightarrow Hom_{UL(\Ih)}(U(\Ih_{Lie})^{s}, M) \rightarrow M \rightarrow Hom_{UL(\Ih)}(U(\Ih_{Lie})^{a} \otimes \Ih, M) \rightarrow Ext^{1}_{UL(\Ih)}(U(\Ih_{Lie})^{s}, M) \rightarrow 0
\]
To obtain the desired isomorphism, we want to relate it to the exact sequence we get from $f$:
\[
0 \longrightarrow Ker(f) \longrightarrow M \longrightarrow Hom(\Ih, HL^{0}(\Ih, M)) \longrightarrow Coker(f) \longrightarrow 0
\]
and conclude by using the 5-lemma.
We can send $M$ onto $M$ via the identity map. We then construct an isomorphism
\begin{equation*}
Hom_{UL(\Ih)}(U(\Ih_{Lie})^{a} \otimes \Ih, M) \longrightarrow Hom(\Ih, HL^{0}(\Ih, M))
\end{equation*}
Notice that since $U(\Ih_{Lie})^{a} \otimes \Ih$ is a quotient of $UL(\Ih) \otimes \Ih$, it is generated, as a $UL(\Ih)$-module, by the elements $1 \otimes h$, for $h \in \Ih$. We can now define a map:
\[
\begin{split}
Hom_{UL(\Ih)}(U(\Ih_{Lie})^{a} \otimes \Ih, M) & \longrightarrow Hom(\Ih, HL^{0}(\Ih, M)) \\
\varphi & \longmapsto \tilde{\varphi}
\end{split}
\]
where $\tilde{\varphi}(h) := \varphi(1 \otimes h)$, for $h \in \Ih$.
The image of $\tilde{\varphi}$ lies in $HL^{0}(\Ih, M)$, for: 
\[
\begin{split}
\left[ \tilde{\varphi}(h), h' \right]_{R} & = \left[ \varphi(1 \otimes h), h' \right]_{R} \\
& = \varphi(r_{h'}.(1 \otimes h)) \\
& = 0
\end{split}
\]
using the fact that $\varphi$ is a $UL(\Ih)$-morphism and the fact that we are considering the $UL(\Ih)$-module $U(\Ih_{Lie})^{a} \otimes \Ih$.

We can then construct its inverse, by:
\[
\begin{split}
Hom(\Ih, HL^{0}(\Ih, M)) & \longrightarrow Hom_{UL(\Ih)}(U(\Ih_{Lie})^{a} \otimes \Ih, M) \\
u & \longmapsto \varphi_{u}
\end{split}
\]
with $\varphi_{u} : \bar{x} \otimes h \mapsto x.u(h)$ where $\bar{x}$ denotes the class of $x \in UL(\Ih)$ in the quotient $U(\Ih_{Lie})^{a}$ (see Section 1).

This yields the following diagram:
\begin{equation*}
\begin{tikzcd}
M \arrow[d, equal] \arrow[r] 
	& Hom_{UL(\Ih)}(U(\Ih_{Lie})^{a} \otimes \Ih, M) \arrow[d, "\sim"] \arrow[r]
		& Ext_{UL(\Ih)}^{1}(U(\Ih_{Lie})^{s}, M) \arrow[d, "(*)"] \arrow[r]
			& 0 \arrow[r]  \arrow[d, equal]
				& 0 \arrow[d, equal]\\
M \arrow[r, "f"]
	& Hom(\Ih, HL^{0}(\Ih, M)) \arrow[r]
		& Coker(f) \arrow[r]
			& 0 \arrow[r]
				& 0
\end{tikzcd}
\end{equation*}
where the arrow $(*) : Ext_{UL(\Ih)}^{1}(U(\Ih_{Lie})^{s}, M) \longrightarrow Coker(f)$ is given by functoriality of the $Coker$.

To conclude, we just need to prove that this diagramm is commutative. It is sufficient to show that it is the case for the square:
\begin{equation*}
\begin{tikzcd}
M \arrow[d, equal] \arrow[r] 
	& Hom_{UL(\Ih)}(U(\Ih_{Lie})^{a} \otimes \Ih, M) \arrow[d, "\sim"] \\
M \arrow[r, "f"]
	& Hom(\Ih, HL^{0}(\Ih, M))
\end{tikzcd}
\end{equation*}
Notice that for the arrow $M \longrightarrow Hom_{UL(\Ih)}(U(\Ih_{Lie})^{a} \otimes \Ih, M)$ we identified
\begin{equation*}
M \simeq Hom_{UL(\Ih)}(UL(\Ih), M)
\end{equation*}
via the map $m \longmapsto (\psi_{m} : u \mapsto u.m)$. This arrow is therefore given by $\psi_{m} \longmapsto \psi_{m} \circ f_{2}$, that is:
\[
\begin{split}
\bar{u} \otimes x \longmapsto \psi_{m}(f_{2}(\bar{u} \otimes x)) & = \psi_{m}(f_{2}(d_{0}(u) \otimes x)) \\
	& = \psi_{m}(f_{1}(u \otimes x)) \\
	& = \psi_{m}(u(l_{x} + r_{x})) \\
	& = u(l_{x} + r_{x}).m
\end{split}
\]
Since $U(\Ih_{Lie})^{a} \otimes \Ih$ is generated as a $UL(\Ih)$-module by the elements $1 \otimes x$ for $x \in \Ih$, it is enough to check the commutativity of the diagram only on these elements. By explicitly writing the maps in question we get:

\begin{equation*}
\begin{tikzcd}
m = \psi_{m} \arrow[d, equal] \arrow[r, mapsto] 
	& (1 \otimes x \longmapsto (l_{x} + r_{x}).m) \arrow[d, mapsto] \\
m \arrow[r, mapsto, "f"] 
	& (x \longmapsto \left[ x, m \right]_{L} + \left[ m, x \right]_{R})
\end{tikzcd}
\end{equation*}
which by \Cref{equi} proves the commutativity of the square and therefore of the diagram.
The 5-lemma then tells us the arrow $(*)$ is an isomorphism, and we obtain the second isomorphism of the proposition.

To get the higher degree isomorphisms, notice that the long exact sequence in cohomology we found earlier goes as follow:
\begin{align*}
\ldots\rightarrow & Ext^{i}_{UL(\Ih)}(UL(\Ih), M) \longrightarrow Ext^{i}_{UL(\Ih)}(U(\Ih_{Lie})^{a} \otimes \Ih, M) \longrightarrow Ext^{i+1}_{UL(\Ih)}(U(\Ih_{Lie}), M) \\
& \longrightarrow Ext^{i+1}_{UL(\Ih)}(UL(\Ih), M) \rightarrow\ldots
\end{align*}
But $UL(\Ih)$ being a free $UL(\Ih)$-module, it is projective, hence 
\begin{align*}
Ext^{i}_{UL(\Ih)}(UL(\Ih), M) & = Ext^{i+1}_{UL(\Ih)}(UL(\Ih), M)\\
& = 0
\end{align*}
and this for all $i$. We thus obtain:
\[
0 \longrightarrow Ext^{i}_{UL(\Ih)}(U(\Ih_{Lie})^{a} \otimes \Ih, M) \longrightarrow Ext^{i+1}_{UL(\Ih)}(U(\Ih_{Lie}), M) \longrightarrow 0
\]
Now, in order to conclude, we use the fact that:
\[
Ext^{i}_{UL(\Ih)}(U(\Ih_{Lie})^{a} \otimes \Ih, M) = Hom( \Ih, Ext^{i}_{UL(\Ih)}(U(\Ih_{Lie})^{a}, M))
\]
which is obtained from the classical Hom/Tens adjunction, and the isomorphism given in Theorem (3.4) of \cite{loday}:
\[
Ext^{i}_{UL(\Ih)}(U(\Ih_{Lie})^{a}, M) \simeq HL^{i}(\Ih, M)
\]
This gives us all the promised isomorphisms, therefore concluding the proof.
\end{proof}

We can now compute the $Ext$ groups in the category of $\Ih$-bimodules and we will see that nontrivial $Ext^{1}$ groups appear whenever the degree shift from \Cref{prop23} is happening.

\begin{theorem} \label{main}
Let $\Ih$ be a finite dimensional simple Leibniz algebra over a field of characteristic zero $k$. All groups $Ext_{UL(\Ih)}^{2}(M, N)$ between simple finite dimensional $\Ih$-bimodules are zero, except $Ext^{2}_{UL(\Ih)}(\mathfrak{M}^{s}, \mathfrak{N}^{a})$, with $\mathfrak{M} \in \{ \mathfrak{Leib(h)}^{\star}, \Ih_{Lie}^{\star} \}$ and $\mathfrak{N} \in \{ \mathfrak{Leib(h)}, \Ih_{Lie} \}$,  which is one dimensional.\\
Moreover, we have that:
\begin{itemize}
\item $Ext^{1}_{UL(\Ih)}(\mathfrak{M}^{s}, k)$, and $Ext^{1}_{UL(\Ih)}(k, \mathfrak{N}^{a})$ are one dimensional, for $\mathfrak{M}\in \{ \mathfrak{Leib(h)}^\star, \Ih_{Lie}^\star \}$ and $\mathfrak{N} \in \{ \mathfrak{Leib(h)}, \Ih_{Lie} \}$;
\item $Ext^{1}_{UL(\Ih)}(M^{s}, N^{a}) \simeq Hom_{U(\Ih_{Lie})}(M, \widehat{N})$, where
\begin{equation*}
\widehat{N} := Coker(h:\; \; N \longrightarrow Hom(\Ih, N))\; \; \; \; h(n)(x):= \left[ x, n\right]_{L}
\end{equation*}
\item All other groups $Ext^{1}_{UL(\Ih)}(M, N)$ between simple finite-dimensional $\Ih$-bimodules $M$ and $N$ are zero.
\end{itemize}  
\end{theorem}

\begin{proof}
We will compute $Ext_{UL(\Ih)}^{*}(M, N)$ for every combination of finite-dimensional $\Ih$-bimodules $M$ and $N$ and reduce it to the Chevalley-Eilenberg cohomology of the simple Lie algebra $\Ih_{\rm Lie}$ in order to conclude.
\begin{itemize}

\item \textit{Case 1:} $M = N = k$ is the trivial $\Ih$-bimodule.\\
We apply \Cref{sspec} to $Y = X = k$. By \Cref{fw}, $HL^{q}( \Ih, k) = 0$ for $q \geq 1$, since $k$ being trivial, it is also symmetric. Therefore, we obtain:
\begin{equation*}
Ext_{UL(\Ih)}^{*}(k, k) \simeq H^{*}(\Ih_{Lie}, k)
\end{equation*}

\item \textit{Case 2:} $M = k$ is the trivial $\Ih$-bimodule, and $N$ is a nontrivial simple symmetric $\Ih$-bimodule. \\
We apply \Cref{sspec} to $Y = k$, and $X = N$. Once again by \Cref{fw}, we get:
\begin{align*}
Ext^{n}_{UL(\Ih)}(k, N^{s}) =  0 &&\text{for $n \geq 1$}
\end{align*}

\item \textit{Case 3:} $M = k$ is the trivial $\Ih$-bimodule, and $N$ is a nontrivial simple antisymmetric $\Ih$-bimodule. \\
We have:
\begin{align*}
HL^{q}(\Ih, N^{a}) & \simeq 0 && \text{for q > 1, by \Cref{fw}} \\
& \simeq Hom_{U(\Ih_{Lie})}(\Ih, N) && \text{for q = 1, by Lemma 1.5 of \cite{FelWag}} \\
& \simeq N && \text{for q = 0, since $N$ is antisymmetric}
\end{align*}
Since $N$ is a nontrivial simple antisymetric $\Ih$-bimodule, it is also a nontrivial simple $\Ih_{Lie}$-module, and therefore $H^{*}(\Ih_{Lie}, N) = 0$ by Whitehead's theorem. Now using \Cref{sspec}, we find:
\begin{align*}
Ext^{*}_{UL(\Ih)}(k, N^{a}) & \simeq H^{*-1}(\Ih_{Lie}, Hom_{U(\Ih_{Lie})}(\Ih, N))\\
& \simeq H^{*-1}(\Ih_{Lie}, k) \otimes Hom_{U(\Ih_{Lie})}(\Ih, N)
\end{align*}
The second isomorphism is given in \cite{Fuchs}, Theorem 2.1.8 pp. 74--75, or in \cite{HS}, Theorem 13.\\
Since $\Ih$ might not be a simple $\Ih_{Lie}$-module, we cannot just apply Schur's lemma to the group $Hom_{U(\Ih_{Lie})}(\Ih, N)$. But this is where the short exact sequence (2) comes in handy. As a sequence of left $\Ih_{Lie}$-modules it actually splits, yielding the decomposition
\begin{equation*}
\Ih = \mathfrak{Leib(h}) \oplus \Ih_{Lie}
\end{equation*}
and since $\Ih$ is a simple Leibniz algebra, this is the decomposition of $\Ih$ into simple $\Ih_{Lie}$-modules.\\
Now, since $N$ is also a simple $\Ih_{Lie}$-module, we get that if $N \simeq \mathfrak{Leib(h)}$ or $N \simeq \Ih_{Lie}$ (as a left $\Ih_{Lie}$-module), then
\begin{equation*}
H^{*-1}(\Ih_{Lie}, k) \otimes Hom_{U(\Ih_{Lie})}(\Ih, N) \simeq H^{*-1}(\Ih_{Lie}, k)
\end{equation*}
If this is not the case, then
\begin{equation*}
H^{*-1}(\Ih_{Lie}, k) \otimes Hom_{U(\Ih_{Lie})}(\Ih, N) \simeq 0
\end{equation*}

\item \textit{Case 4:} $M$ is a nontrivial simple antisymmetric $\Ih$-bimodule, and $N$ is a simple symmetric $\Ih$-bimodule.\\
Using \Cref{fw}, we have $HL^{q}(\Ih, N^{s}) = 0$ for $q \geq 1$. Moreover, because $HL^{0}(\Ih, N^{s}) = N^{\Ih}$ is a trivial $\Ih$-bimodule, and since we can identify $Hom(M, HL^{0}(\Ih, N^{s})) \simeq M^{\star} \otimes N^{\Ih}$ with the direct sum of dim($N^{\Ih}$) copies of $M^{\star}$ we find that $H^{p}(\Ih_{Lie}, Hom(M, HL^{0}(\Ih, N^{a}))) \simeq H^{p}(\Ih_{Lie}, M^{\star}) \oplus ... \oplus H^{p}(\Ih_{Lie}, M^{\star}) = 0$, since $M$ being a simple nontrivial $\Ih_{Lie}$-module, so is $M^{\star}$. Thus yielding:
\begin{equation*}
Ext^{*}_{UL(\Ih)}(M^{a}, N^{s}) = 0
\end{equation*}

\item \textit{Case 5:} $M$ is a nontrivial simple antisymmetric representation, and $N$ is simple and antisymmetric.\\
Here, \Cref{fw} apply again, and we have that $HL^{q}(\Ih, N^{a}) \neq 0$ only when $q \in \{0, 1\}$. We check that $HL^{1}(\Ih, N^{a})$ is a trivial left $\Ih$-module. By definition of the chain complex defining Leibniz cohomology, we have that $CL^{1}(\Ih, N^{a}) = Hom(\Ih, N)$. Now for a morphism $\varphi \in Hom(\Ih, N)$ to be annihilated by the differential $dL^{1}$ means satisfying:
\begin{align*}
dL^{1}\varphi(x, y) := \left[ x, \varphi(y) \right]_{L} - \varphi(\left[ x, y \right]) = 0 && \forall x, y \in \Ih, 
\end{align*}
which is exactly to say that the left action of $\Ih$ on the module $Hom(\Ih, N)$ is trivial. Therefore, the same arguments used in Case 4 still apply, and we get that $E^{pq}_{2} = 0$ for $q > 0$, and:
\begin{equation*}
Ext^{*}_{UL(\Ih)}(M^{a}, N^{a}) = 0
\end{equation*}

\item \textit{Case 6:} $M$ is a nontrivial simple symmetric representation, and $N = k$ is the trivial $\Ih$-bimodule.\\
We apply \Cref{prop23} to $k$ to find:
\begin{align*}
Ext^{i}_{UL(\Ih)}((U(\Ih_{Lie})^{s}, k) & \simeq 0 && \text{if i > 1} \\
& \simeq \Ih^\star && \text{if i = 1} \\
& \simeq k && \text{if i = 0}
\end{align*}
because since in this case, the map $f$ in \Cref{prop23} is zero. We can now plug this in the second spectral sequence of \Cref{sspec}, with $X = k$, and $Z = M$, to obtain:
\begin{align*}
Ext^{*}_{UL(\Ih)}(M^{s}, k) & \simeq H^{*-1}(\Ih_{Lie}, Hom(M, \Ih^\star)) \\
& \simeq H^{*-1}(\Ih_{Lie}, k) \otimes Hom_{U(\Ih_{Lie})}(M, \Ih^\star)
\end{align*}

Using the same arguments as in Case 3, we get that if $M \simeq \mathfrak{Leib(h)}^\star$ or $M \simeq \Ih_{Lie}^\star$, then
\begin{equation*}
H^{*-1}(\Ih_{Lie}, k) \otimes Hom_{U(\Ih_{Lie})}(M, \Ih^\star) \simeq H^{*-1}(\Ih_{Lie}, k)
\end{equation*}
If this is not the case, then
\begin{equation*}
H^{*-1}(\Ih_{Lie}, k) \otimes Hom_{U(\Ih_{Lie})}(M, \Ih^\star) \simeq 0
\end{equation*}

\item \textit{Case 7:} $M$ and $N$ are both simple nontrivial symmetric $\Ih$-bimodules.\\
Applying \Cref{prop23} to $N^{s}$, and because $N$ is a symmetric $\Ih$-bimodule, we find that 
\begin{align*}
Ext^{i}_{UL(\Ih)}(U(\Ih_{Lie})^{s}, N^{s}) & \simeq 0 && \text{i $\geq$ 1} \\
& \simeq N && \text{if i = 0}
\end{align*}
Now using the second spectral sequence of \Cref{sspec}, we get:
\begin{align*}
Ext^{*}_{UL(\Ih)}(M^{s}, N^{s}) & \simeq H^{*}(\Ih_{Lie}, Hom(M, N)) \\
& \simeq H^{*}(\Ih_{Lie}, k) \otimes Hom_{\Ih_{Lie}}(M, N)
\end{align*}
Once again, since $M$ and $N$ are simple $\Ih_{Lie}$-modules, this vector space is nonzero only if $M \simeq N$, in which case it is isomorphic to $H^{*}(\Ih_{Lie}, k)$.

\item \textit{Case 8:} $M$ is a simple nontrivial symmetric $\Ih$-bimodule, and $N$ is a simple nontrivial antisymmetric $\Ih$-bimodule.\\
By \Cref{prop23}, we have:
\begin{align*}
Ext^{i}_{UL(\Ih)}(U(\Ih_{Lie})^{s}, N^{a}) & \simeq 0 && \text{for i > 2}\\
& \simeq Hom(\Ih, Hom_{U(\Ih_{Lie})}(\Ih, N)) && \text{for i = 2}\\
& \simeq Coker(h) && \text{for i = 1}\\
& \simeq Ker(h) && \text{for i = 0}
\end{align*}
The map $h$ appearing here (defined in the statement of the theorem) is due to the fact that $N$ is an antisymmetric $\Ih$-bimodule. Moreover, since $N$ is supposed to be nontrivial and $h$ is a $\Ih$-module homomorphism, $Ker(h) = 0$. Therefore we have that $E^{pq}_{2} = 0$ for $q > 2$ and $q = 0$. For the remaining values of $q$, we have isomorphisms
\begin{align*}
E^{p1}_{2} & \simeq H^{p}(\Ih_{Lie}, Hom(M, \widehat{N}))\\
& \simeq H^{p}(\Ih_{Lie}, k) \otimes Hom_{U(\Ih_{Lie})}(M, \widehat{N})
\end{align*}
and
\begin{align*}
E^{p2}_{2} & \simeq H^{p}(\Ih_{Lie}, Hom(M, Hom(\Ih, Hom_{U(\Ih_{Lie})}(\Ih, N))))\\
& \simeq H^{p}(\Ih_{Lie}, k) \otimes Hom_{U(\Ih_{Lie})}(M, Hom(\Ih, Hom_{U(\Ih_{Lie})}(\Ih, N)))
\end{align*}
The first isomorphism tells us that $Ext^{1}_{UL(\Ih)}(M^{s}, N^{a}) \simeq Hom_{U(\Ih_{Lie})}(M, \widehat{N})$.\\
To use the second isomorphism, we need to proceed as in Case 6, since $\Ih$ is not a priori a simple $\Ih_{Lie}$-module, although with some more cases.

	\begin{itemize}
	\item If $N \not\simeq \mathfrak{Leib(h)}$ or $N \not\simeq \Ih_{Lie}$:\\
	Then $Hom_{U(\Ih_{Lie})}(\Ih, N) \simeq 0$, yielding $E^{p2}_{2} = 0$
	\item If $N \simeq \mathfrak{Leib(h)}$ or $N \simeq \Ih_{Lie}$:\\
	Then $Hom_{U(\Ih_{Lie})}(\Ih, N) \simeq k$, and we have
	\begin{align*}
		E^{p2}_{2} & \simeq H^{p}(\Ih_{Lie}, Hom(M, \Ih^{\star}))\\
		& \simeq H^{p}(\Ih_{Lie}, k) \otimes Hom_{U(\Ih_{Lie})}(M, \Ih^{\star})
	\end{align*}
	Now we need to do the same work for $Hom_{U(\Ih_{Lie})}(M, \Ih^{\star})$. Since $\Ih$ is a simple Leibniz algebra, $\Ih_{Lie}$ is a simple Lie algebra and is isomorphic as an $\Ih_{Lie}$-module to its dual $\Ih_{Lie}^{\star}$ via the Killing form. Moreover, the exactness of the functor $Hom(-, k)$ gives us the short exact sequence
	\begin{equation*}
	0 \longrightarrow \Ih _{Lie}^{\star} \longrightarrow \Ih^{\star} \longrightarrow  \mathfrak{Leib(h)}^{\star} \longrightarrow 0
	\end{equation*}
	and the decomposition of $\Ih^{\star} = \mathfrak{Leib(h})^{\star} \oplus \Ih_{Lie}^{\star}$ as a left $\Ih_{Lie}$-module, since the $\mathfrak{Leib(h)}$ is a simple $\Ih_{Lie}$-module and therefore so is its dual.

	We therefore are in one of the following cases:
			\begin{itemize}
			\item If $M \not\simeq \mathfrak{Leib(h)}^{\star}$ or $M \not\simeq \Ih_{Lie}^{\star}$:\\
			Then $Hom_{U(\Ih_{Lie})}(M, \Ih^{\star}) \simeq 0$, and $E^{p2}_{2} = 0$.
			\item If $M \simeq \mathfrak{Leib(h)}^{\star}$ or $M \simeq \Ih_{Lie}^{\star}$:\\
			Then $Hom_{U(\Ih_{Lie})}(M, \Ih^{\star}) \simeq k$, and we get
			\begin{align*}
			E^{p2}_{2} & \simeq H^{p}(\Ih_{Lie}, k) \otimes Hom_{U(\Ih_{Lie})}(M, \Ih^{\star})\\
			& \simeq H^{p}(\Ih_{Lie}, k)
			\end{align*}
			\end{itemize}
	\end{itemize}	
\end{itemize}
In order to get the promised vanishing of the Ext groups, we just use the fact that, since $\Ih_{Lie}$ is simple:
\begin{equation*}
H^{1}(\Ih_{Lie}, k) \simeq H^{2}(\Ih_{Lie}, k) \simeq 0
\end{equation*}
and this concludes our proof.
\end{proof}

This theorem also allows us to compute the $Ext$ dimension of the category, denoted by $L(\mathfrak{h})$ in \cite{LP}, of finite dimensional bimodules (but here in the left setting). First we need to make sure that the $Ext$ groups are well defined in this category. To do so we will use more general results from Category Theory.

Notice that $L(\mathfrak{h})$ is an {\it essentially small} abelian category. This means by definition that 
the class of isomorphism classes of objects is a set. This property implies that every one of its categories of fractions exists and is essentially small as well (see for example Proposition 5.2.2 of \cite{B}). Therefore the derived category of $\mathfrak{L(h)}$ is well defined.

All there is to do now is to see that one can relate the $Ext$ groups with morphisms in the derived category. For this, we refer the reader to \textsection 5 of Chapter III, and \textsection 6.14 of Chapter III in \cite{GM}. In \textsection 5 the authors define the $Ext$ functor in terms of the $Hom$ in the derived category, and they show in \textsection 6.14 that it is equivalent to the derived functor definition.

Therefore, the $Ext$ groups are well defined in the category $L(\mathfrak{h})$ and we can now state the following corollary to \Cref{main}:

\begin{corollary}
For $i \in \{ 0, 1, 2 \}$, the natural transformation $Ext^{i}_{L(\mathfrak{h})} \longrightarrow Ext^{i}_{UL(\Ih)}$ induced by the inclusion functor from $L(\mathfrak{h})$ to the category of $UL(\Ih)$-modules is an isomorphism. Moreover, since $Ext^{i}_{L(\mathfrak{h})} = 0$ for $i > 2$, the $Ext$ dimension of $L(\mathfrak{h})$ is $2$.
\end{corollary}

\begin{proof}
For finite dimensional bimodules $M$ and $N$, it is clear that $Hom_{L(\mathfrak{h})}(M, N) = Hom_{UL(\Ih)}(M, N)$, as well as $Ext^{1}_{L(\mathfrak{h})}(M, N) = Ext^{1}_{UL(\Ih )}(M, N)$. This gives us the $i \in \{0, 1 \} $ cases.

For $i=2$, because of the short exact sequence (3), it is enough to consider the case when $M$ and $N$ are simple objects. In this case \Cref{main} tells us that $Ext^{2}_{UL(\Ih)}(M^{s}, N^{a}) \neq 0$ only when $M \in \{ \mathfrak{Leib(h)}^{\star}, \Ih_{Lie}^{\star} \}$ and $N \in \{ \mathfrak{Leib(h)}, \Ih_{Lie} \}$. When this Ext group is zero, then $Ext^{2}_{L(\mathfrak{h})} \longrightarrow Ext^{2}_{UL(\Ih)}$ is obviously an isomorphism.

When $Ext^{2}_{UL(\Ih)}(M^{s}, N^{a}) \neq 0$, then it is one dimensional. This means that we only have to produce a non trivial two-fold extension $0 \longrightarrow N^{a} \longrightarrow E \longrightarrow F \longrightarrow M^{s} \longrightarrow 0$ with $\dim(E),\dim(F)<\infty$ to conclude.

Since $Ext^{2}_{UL(\Ih)}(M^{s}, N^{a})$ is one dimensional, then we can select one of its generators. It is the equivalence class of an exact sequence 
\begin{equation*}
\begin{tikzcd}
0 \arrow[r] & N^{a} \arrow[r] & E \arrow[r, "\varphi"] & F \arrow[r] & M^{s} \arrow[r] & 0
\end{tikzcd}
\end{equation*}
We can split this sequence into two exact sequences of length $1$, $0 \longrightarrow N^{a} \longrightarrow E \longrightarrow Im(\varphi) \longrightarrow 0$ and $0 \longrightarrow Im(\varphi) \longrightarrow F \longrightarrow M^{s} \longrightarrow 0$.

Since the sequence represents a generator of $Ext^{2}_{UL(\Ih)}(M^{s}, N^{a})$, the $1$-fold exact sequences cannot split, i.e. represent trivial classes in the corresponding $Ext^{1}_{UL(\Ih)}$. But by \Cref{main} again, for this to be the case, they must be in $Ext^{1}_{UL(\Ih)}(M^{s}, k )$ and $Ext^{1}_{UL(\Ih)}(k, N^{a})$ respectively, since the bimodule $Im(\varphi)$ must be both symmetric and antisymmetric (else the extensions are split, again using \Cref{main}).

But we know that $Ext^{1}_{UL(\Ih)}(M^{s}, k ) = Ext^{1}_{L(\mathfrak{h})}(M^{s}, k )$, and $Ext^{1}_{UL(\Ih)}(k, N^{a}) = Ext^{1}_{L(\mathfrak{h})}(k, N^{a})$. This means that $E$ and $F$ must be of finite dimension, allowing us to conclude.

The last step is to show that $Ext^{i}_{L(\mathfrak{h})}(M, N) = 0$ for $i \geq 3$. Note that this is not true for the whole category of $UL(\Ih)$-bimodules. But we have seen in Theorem \ref{main} that all higher $Ext$ groups come from higher Chevalley-Eilenberg cohomology 
$H^*(\Ih_{Lie},k)=Ext^*_{U\Ih_{Lie}}(k,k)$ of the simple Lie algebra $\Ih_{Lie}$. We claim that these $Ext$ groups vanish for $*>0$ in the subcategory (of finite dimensional $\Ih_{Lie}$-modules of the subcategory) of finite dimensional $UL(\Ih)$-bimodules. Indeed, for a simple Lie algebra $\Ih$, the category of finite-dimensional $\Ih_{Lie}$-modules is semisimple by Weyl's theorem (see for example Theorem 7.8.11 in \cite{Wei}). 
\end{proof}

\end{document}